\documentclass[12pt]{article}

\usepackage[english]{babel}
\usepackage{amsmath}
\usepackage{amsfonts}
\usepackage{amsthm,amssymb,latexsym}
\usepackage{graphics,graphicx}
\usepackage{epsfig}
\usepackage{epsf}
\usepackage{eepic,epic}

\usepackage{authblk}

\long\def\symbolfootnote[#1]#2{\begingroup
\def\thefootnote{\fnsymbol{footnote}}\footnote[#1]{#2}\endgroup}

\newtheorem{prop}{Proposition}
\newtheorem{obs}{Observation}
\newtheorem{lemma}[prop]{Lemma}

\newtheorem{theorem}[prop]{Theorem}

\theoremstyle{definition}

  \title{On the representation number of a crown graph}

\author[1]{Marc Glen}
\author[1]{Sergey Kitaev}
\author[2]{Artem Pyatkin}

\affil[1]{\small Department of Computer and Information Sciences, University of Strathclyde, 
26 Richmond Street, Glasgow G1 1XH, UK}
\affil[2]{Sobolev Institute of Mathematics, Novosibirsk State University, 630090 Novosibirsk, Russia}

\date{}

\begin{document}

\maketitle

\abstract{A graph $G=(V,E)$ is word-representable if there exists a word $w$ over the alphabet $V$ such that letters $x$ and $y$ alternate in $w$ if and only if $xy$ is an edge in $E$. It is known that any word-representable graph $G$  is $k$-word-representable for some $k$, that is, there exists  a word $w$ representing $G$ such that each letter occurs exactly $k$ times in $w$. The minimum such $k$ is called $G$'s representation number. 

A  crown graph $H_{n,n}$ is a graph obtained from the complete bipartite graph $K_{n,n}$ by removing a perfect matching. In this paper we show that for $n\geq 5$, $H_{n,n}$'s representation number is $\lceil n/2 \rceil$. This result not only provides a complete solution to the open Problem 7.4.2 in \cite{KL}, but also gives a negative answer to the question raised in Problem 7.2.7 in \cite{KL} on 3-word-representability of bipartite graphs. As a byproduct we obtain a new example of a graph class with a high representation number. \\

}

\section{Introduction}

A graph $G=(V,E)$ is word-representable if there exists a word $w$ over the alphabet $V$ such that letters $x$ and $y$ alternate in $w$ if and only if $xy$ is an edge in $E$. For example, the cycle graph on 4 vertices labeled by 1, 2, 3 and 4 in clockwise direction can be represented by the word 14213243. 

There is a long line of research on word-representable graphs, which is summarised in the recently published book \cite{KL}. The roots of the theory of word-representable graphs are in the study of the celebrated Perkins semigroup \cite{KS} which has played a central role in semigroup theory since 1960, particularly as a source of examples and counterexamples. 

It was shown in \cite{KP} that if a graph $G$ is word-representable then it is {\em $k$-word-representable} for some $k$, that is, $G$ can be represented by a $k$-{\em uniform} word $w$, i.~e. a word containing $k$ copies of each letter. In such a context we say that $w$ {\em $k$-represents} $G$. For example, the cycle graph on 4 vertices mentioned above can be 2-represented by the word 14213243. Thus, when discussing word-representability, one need only consider  $k$-uniform words. The nice property of such words is that any cyclic shift of a $k$-uniform word represents the same graph  \cite{KP}. The minimum $k$ for which a word-representable graph $G$ is $k$-word-representable is called $G$'s {\em representation number}.

The following observation trivially follows from the definitions.

\begin{obs}\label{propY} The class of complete graphs coincides with the class of $1$-word-representable graphs. In particular, the complete graph's representation number is $1$.\end{obs}

\subsection{Representation of crown graphs}
A {\it crown graph} (also known as a {\it cocktail party graph}) $H_{n,n}$ is a graph obtained from the complete bipartite graph $K_{n,n}$ by removing a perfect matching. Formally, 
$V(H_{n,n})=\{1,\ldots, n,1',\ldots, n' \}$ and $E(H_{n,n})=\{ ij' \ | i\ne j \}$. First four examples of such graphs are presented in Figure~\ref{crown-pic}.

\begin{figure}[ht]
\begin{center}
\includegraphics[scale=0.6]{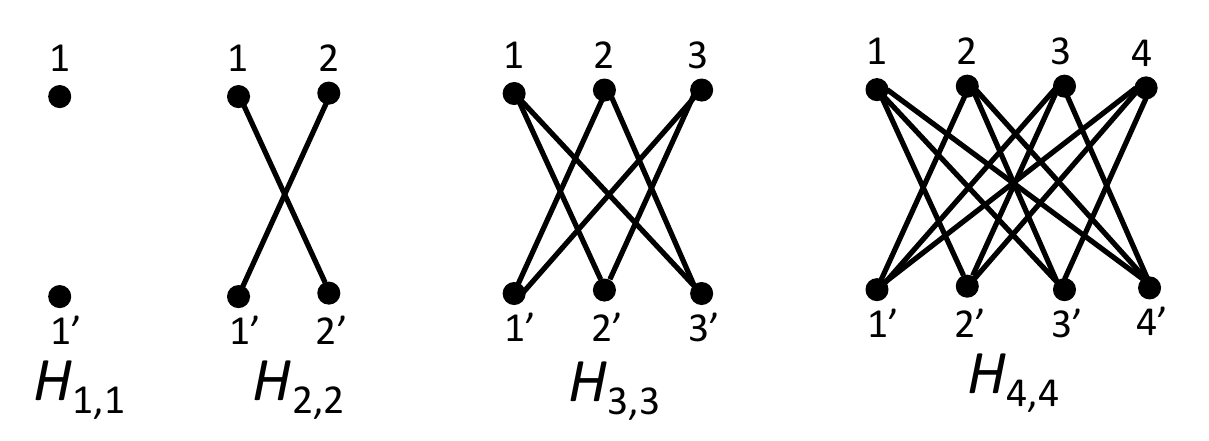}
\end{center}
\vspace{-20pt}
\caption{The crown graph $H_{n,n}$ for $n=1,2,3,4$}
\label{crown-pic}
\end{figure}

Crown graphs are of special importance in the theory of word-representable graphs. More precisely, they appear in the construction of graphs requiring long words representing them \cite{HKP}.  Note that these graphs also appear in the theory of partially ordered sets as those defining partial orders that require many linear orders to be represented. 

Each crown graph, being a bipartite graph, is a {\em comparability graph} (that is, a transitively orientable graph), and thus it can be represented by a concatenation of permutations \cite{KS}. Moreover, it follows from \cite{HKP}, and also is discussed in Section 7.4 in \cite{KL}, that $H_{n,n}$ can be represented as a concatenation of $n$ permutations but it cannot be represented as a concatenation of a fewer permutations. Thus, the representation number of $H_{n,n}$ is at most $n$.  See Table~\ref{ex-k-repr-Hkk} (appearing in \cite{KL}) for the words representing the graphs in Figure~\ref{crown-pic} as concatenation of permutations.

\begin{table}
\begin{center}
\begin{tabular}{c|c}
$n$ & representation of $H_{n,n}$ by concatenation of $n$ permutations\\
\hline
1 & $11'1'1$\\
\hline
2 & $12'21'21'12'$\\
\hline
3 & $123'32'1'132'23'1'231'13'2'$\\ 
\hline
4 & $1234'43'2'1'1243'34'2'1'1342'24'3'1'2341'14'3'2'$\\ 
\hline
\end{tabular}
\caption{Representing $H_{n,n}$ as a concatenation of $n$ permutations}\label{ex-k-repr-Hkk}
\end{center}
\end{table}

It was noticed in \cite{K} that, for example, $H_{3,3}$ can be represented using two copies of each letter as $3'32'1'132'23'1'231'1$ (as opposed to three copies used in Table~\ref{ex-k-repr-Hkk} to represent it) if we drop the requirement to represent crown graphs as concatenation of permutations.  On the other hand, $H_{4,4}$ is the three-dimensional cube, which is the {\em prism graph} Pr$_4$, so that $H_{4,4}$ is 3-word-representable by Proposition~15 in \cite{KP}, while four copies of each letter are used in Table~\ref{ex-k-repr-Hkk} to represent this graph. Note that $H_{4,4}$ is not 2-word-representable by Theorem~18 in \cite{K}. 

These observations led to Problem 7.4.2 on page 172 in \cite{KL} essentially asking to find the representation number of a crown graph $H_{n,n}$. A relevant Problem 7.2.7 on page 169 in \cite{KL} asks whether each bipartite graph is 3-word-representable. When we started to investigate these problems, we established that both $H_{5,5}$ and $H_{6,6}$ are 3-word-representable, which not only suggested that the representation number of a crown graph could be the constant 3, but also that any bipartite graph could be 3-word-representable since crown graphs seem to be the most difficult amoung them to be represented. 

In this paper we completely solve the former problem (Problem 7.4.2) and provide the negative answer to the question in the latter problem (Problem 7.2.7) by showing that if $n\ge 5$ then the crown graph  $H_{n,n}$, being a bipartite graph, is $\lceil n/2 \rceil$-representable (see Theorem~\ref{thm1}). Thus, crown graphs are another example of a graph class with high representation number. Note that non-bipartite graphs obtained from crown graphs by adding an all-adjacent vertex require roughly twice as long words representing them (see Section 4.2.1 in \cite{KL}).

\subsection{Organization of the paper and some definitions}

The paper is organized as follows. In Section~\ref{sec-lower}, we find a lower bound for the representation number of $H_{n,n}$, while in Section~\ref{sec-upper} we provide a construction of words representing $H_{n,n}$ that match our lower bound. Finally, in Section~\ref{conclusion} we provide some concluding remarks including directions for further research.

We conclude the introduction with a number of technical definitions to be used in the paper. 

A {\em factor} of a word is a number of consecutive letters in the word.  For example, the set of all factors of the word 1132 of length at most 2 is $\{1,2,3,11,13,32\}$. A {\em subword} of a word is a subsequence of letters in the word. For instance, 56, 5212 and 361 are examples of subwords in 3526162. The subword of a word $w$ {\em induced} by a set $A$ is obtained by removing all elements in $w$ not belonging to $A$. For example, if $A=\{2,4,5\}$ then the subword of 223141565 induced by $A$ is  22455.

For a vertex $v$ in a graph $G$ denote by $N(v)$ the neighbourhood of  $v$, i.~e. the set of vertices adjacent to $v$. Clearly, if a graph is bipartite then the neighbourhood of each vertex induces an independent set.

\section{A lower bound for the representation number of $H_{n,n}$}\label{sec-lower}

For a word $w$, let $l(w)$ and $r(w)$ be its first and last letters, respectively.

Let $w$ be a word that $k$-represents a graph $G=(V,E)$.  A subset $A\subseteq V$ is {\it splittable} if there is a cyclic shift of the word $w$ such that the subword induced by the set $A$ has the form $P_1\cdots P_k$ where each $P_i$ is a permutation of $A$. For a splittable set $A$, a {\it canonical shift} of $w$, with respect to $A$, is a cyclic shift of $w$ that puts 
$l(P_1)$ at the beginning of the word. Note that up to renaming permutations there is a unique canonical shift.
The following proposition gives an example of a splittable set.

\begin{prop}\label{propX} For any vertex $v\in V$ in a word-representable graph $G=(V,E)$, the set $A=N(v)$ is splittable. \end{prop}

\begin{proof} Consider a cyclic shift of a word $w$ $k$-representing $G$ that puts $v$ at the beginning of the word. Then between any two occurrences of $v$ (and after the last one) each letter from $A$ occurs exactly once, i.e.\ the subword induced by $A$ is a concatenation of permutations. Note, however, that this shift is not canonical with respect to $A$. \end{proof}

For a letter $x$, denote by $x_i$ its $i$-th occurrence in a word $w$ (from left to right). We write $x_i<y_j$ if the $i$-th occurrence of $x$ is to the left of the $j$-th occurrence of $y$ in $w$. Clearly, if $A$ is splittable, then for every $a,b\in A$ and for all $i,j$ such that $1\le i<j\le k$, we have $a_i<b_j$.

\begin{lemma}\label{lem1}
Let a word $w$ $k$-represent a graph $G=(V,E)$ and $A\subset V$ be a splittable set. Further, let $a,b\in A$, $x\not\in A$ and $ax,bx\in E$. If in a canonical shift of $w$ $a_1<x_1<b_1$ then 
$ab\in E$. \end{lemma}

\begin{proof} Let $a_1<x_1<b_1$. Since $A$ is splittable, $b_i<a_{i+1}$ for each $i$. Since both $a$ and $b$ are adjacent to $x$, we have $a_i<x_i<b_i$ for every $i=1,\ldots, k$. Therefore, $a$ and $b$ alternate in $w$ and must be adjacent in $G$. \end{proof}

\begin{lemma}\label{lem2} If  $n\ge 5$ then in any word $w$ $k$-representing $H_{n,n}$ the set $A=\{ 1,\ldots, n\}$ is splittable. \end{lemma}

\begin{proof} By Proposition~\ref{propX}, the set $B:=N(1')=\{2,\ldots, n\}$ is splittable, i.e. there is a cyclic shift of $w$ in which the letters of $B$ form the subword $P'_1\cdots P'_k$, where $P'_i$ is a permutation of $B$. Let a canonical shift of $w$ with respect to $B$ be $P_1I_1\cdots P_kI_k$, where for $i=1,\ldots, k$, the factor  $P_i$ begins at $l(P'_i)$ and ends at $r(P'_i)$, and $I_i$s are (possibly empty) factors lying between $r(P'_i)$ and $l(P'_{i+1})$. We begin by proving the following fact.  \\

\noindent
{\bf Claim 1.} For every $t\geq 1$ and $i\geq 1$ such that $i+t-1\leq k$, the factor $U=P_iI_i\cdots I_{i+t-2}P_{i+t-1}$ of $w$ contains at most $t$ copies of the letter~1. \\

\noindent
{\bf Proof of Claim 1.} Indeed, suppose not. Using a cyclic shift of $w$ if necessary, without loss of generality  we can assume that in a problematic case $i=1$. First consider the case $t=1$. That is, we assume that $P_1$ contains at least two 1s. 
Let $a=l(P_1), b= r(P_1)$ and $x\in V\setminus \{ 1,a,b,1',2',\ldots,n'\}$. Recall that $a,b$ belong to the splittable set $B$. Then the letter $x'$ occurs exactly once between any two consecutive occurrences of 1, in particular, between the first two occurrences. Hence we have  $a_1<x'_1<b_1$. Since both $a$ and $b$ are adjacent to $x'$, it follows from Lemma~\ref{lem1} that $ab\in E$, contradiction. 

Now let $t\ge 2$. Let $a=l(P_1), b= r(P_1), c=r(P_t)$ and $x\in V\setminus  \{ 1,a,b,c,1',2',\ldots,n'\}$ (recall that $n\ge 5$), and suppose that there are at least $t+1$ occurrences of 1 between $a$ and $c$. Note that $a\neq b$, but it is possible that $a=c$ or $a=b$. Since $1x'\in E$, there must be at least $t$ occurrences of $x'$ between $a$ and $c$. By Lemma~\ref{lem1}, no $x'$ can appear between $a_1$ and $b_1$. However, $c$ appears exactly once between $a_1$ and $b_1$ (possibly coinciding with one of them) because $P_1$ contains the permutation $P'_1$ over $B$ as a subword. Moreover, there are exactly $t$ occurrences of $c$  in $U$.
Therefore, the subword of $U$ induced by $c$ and $x'$ 
starts and ends with $c$ and contains at least $t$ copies of $x'$. Clearly, such subword cannot be alternating, which  contradicts $cx'\in E$. Claim~1 is proved. \qed \\

It follows from Claim 1 that each $P_iI_i$ contains at most two 1s, since $P_iI_iP_{i+1}$ contains at most two 1s for $1\leq i\leq k-1$. If each of  $P_iI_i$ contains exactly one 1 then add 1 to each $P'_i$ 
to obtain the concatenation of permutations for the set $A$ showing that it is splittable.
Otherwise, some  $P_iI_i$ must contain at least two 1s. Without loss of generality, $i=1$
(otherwise, we can apply a cyclic shift and rename the permutations). By Claim 1 applied to $P_1$ and $P_1I_1P_2$, at least one of 1s must be in $I_1$ and $P_2$ contains no 1s. So, add the first occurrence of 1 to $P'_1$ and the second one to $P'_2$. If $I_2$ contains no 1s we apply the same arguments to the word obtained from $w$ by removing the factor $P_1I_1P_2I_2$. Otherwise, again by Claim 1 applied to $P_1I_1P_2I_2P_3$, $I_2$ has one 1, $P_3$ has no 1 and we add this 1 to $P'_3$ and continue in the same way showing that $w$ contains as a subword a  concatenation of permutations over $A$, and thus $A$ is splittabe. \end{proof}

\begin{lemma}\label{lem3} Let $n\ge 5$ and  $w$ $k$-represents $H_{n,n}$. Also, let $P'_1\cdots P'_k$ be a subword of (a cyclic shift of) $w$ that is a concatenation of permutations over $A=\{ 1,\ldots, n\}$ (existing by Lemma~\ref{lem2}). Then for every $a\in A$ there is $j\in \{1,\ldots, k\}$ such that $a=l(P'_j)$ or $a=r(P'_j)$.\end{lemma}

\begin{proof}
Assume that the letter 1 is never the first or the last letter of any permutation $P'_j$. Consider a canonical shift of $w$ for the set $A$ and define the subwords $P_i$ and $I_i$ for permutations $P'_i$ in the same way as in the proof of Lemma~\ref{lem2}.  Since $l(P_1)\ne 1$ and $r(P_1)\ne 1$, no $1'$ can appear between 
$l(P_1)$ and $r(P_1)$ by Lemma~\ref{lem1}. This is true for any $P_i$ since we can apply a cyclic shift and rename $P_i$ and $P_1$. Moreover, no $I_i$ can have two or more $1'$s, or no $1'$s at all, because otherwise $1'$ would  not be adjacent to the vertices in $\{2,\ldots,n\}$.

But since each $P_j$ for $j=1,\ldots, k$ contains one $1$, the letters 1 and $1'$ alternate in $w$, i.e. the vertices 1 and $1'$ must be adjacent, contradicting the definition of $H_{n,n}$. \end{proof}

\begin{theorem}\label{thm1} For $n\geq 1$, the representation number of $H_{n,n}$ is at least $\lceil n/2 \rceil$.
\end{theorem}

\begin{proof} We consider three cases.
\begin{itemize}
\item The statement is trivial for $n=1,2$ since each graph requires at least one copy of each letter to be represented. 
\item None of $H_{n,n}$'s is a complete graph, and thus, by Observation~\ref{propY}, the statement is true for $n=3,4$. 
\item Let $n\ge 5$. Since the set  $A=\{ 1,\ldots, n\}$ is splittable by Lemma~\ref{lem2}, and each of its $n$ letters must be the first or the last letter of some permutation $P'_j$ for $j=1,\ldots,k$ by Lemma~\ref{lem3}, we have the inequality $2k\ge n$. Since $k$ is an integer, we obtain the bound  $k\ge \lceil n/2 \rceil$. 
\end{itemize} \end{proof}

\section{An upper bound for the representation number of $H_{n,n}$}\label{sec-upper}

In this section we provide a construction that shows that the bound in Theorem~\ref{thm1} is tight for all $n$ except $n=1,2,4$. We need the following auxiliary fact.

\begin{lemma}\label{lem4} If $n=2k\ge 6$ then for every partition of the set  $A=\{ 1,\ldots, 2k\}$ into $k$ pairs $(a_1,b_1),\ldots, (a_k,b_k)$ there exist permutations $P(a_1,b_1), \ldots,  P(a_k,b_k)$ such that:

\begin{enumerate}
\item $l(P(a_i,b_i))=a_i, r(P(a_i,b_i))=b_i$ for each $i=1,\ldots, k$, and
\item For every $x,y\in A$ there are $i,j$ such that $x<y$ in $P(a_i,b_i)$ and $y<x$ in $P(a_j,b_j)$.
\end{enumerate}
\end{lemma}

\begin{proof} Let $P$ be an arbitrary permutation over the set $A\setminus \{a_1,a_2, b_1, b_2\}$, $Rev(P)$ be obtained from $P$ by writing it in the reverse order, $P'$ be an arbitrary permutation over the set $A\setminus \{a_1,a_2, a_3,b_1, b_2,b_3\}$ and for each $i=4,\ldots, k$ let $P_i$ be an arbitrary permutation over the set $A\setminus \{a_i, b_i\}$. Define the sought permutations as follows: $P(a_1,b_1)=a_1b_2Pa_2b_1,\ P(a_2,b_2)=a_2b_1Rev(P)a_1b_2,\  P(a_3,b_3)=a_3b_2a_1P'b_1a_2b_3$ and $P(a_i,b_i)=a_iP_ib_i$ for each $i=4,\ldots, k$. It is straightforward to verify that both requirements of the lemma hold for these permutations. \end{proof}

Note that for $n\in \{2,4\}$ Lemma \ref{lem4} is not true.

\begin{theorem}\label{thm2} If $n\ge 5$ then the crown graph  $H_{n,n}$ is $\lceil n/2 \rceil$-representable. \end{theorem}

\begin{proof} It is sufficient to prove the theorem only for $n=2k$, $k\geq 3$, because the case of $n=2k-1$ is obtained from the case of $n=2k$ by removing all occurrences of the letters $2k$ and $(2k)'$ from the respective word. First, consider the following $k$-uniform word, where the permutations $P(x,y)$s are defined in Lemma~\ref{lem4} and $P(x',y')$s are obtained from these by adding primes. 
$$w'=P(1,2)P(2',3')P(3,4)P(4',5')\cdots P(n-1,n)P(n',1').$$ It follows from the property 2 in Lemma~\ref{lem4}  that  $w'$ represents the complete bipartite graph $K_{n,n}$. Shift $w'$ cyclicly one position to the left to obtain the word $w''$ where for every even $i$ there is exactly one occurrence of the factor $ii'$ and for every odd $i$ there is exactly one occurrence of the factor $i'i$. Let $w$ be the word obtained from $w''$ by switching $i$ and $i'$ in each of these factors. This operation makes the subword induced by $i$ and $i'$ non-alternating (thus removing the edges $ii'$ in $K_{n,n}$) but does not affect any other alternations in the word. Therefore, $w$ $k$-represents $H_{n,n}$, as desired.\end{proof}

Note that for $n<4$ the graph $H_{n,n}$ is 2-word-representable, which is given by the words $w_1=11'1'1$, $w_2= 12'21'21'12'$ and $w_3=12'3'123'1'231'2'3$, respectively (see pages 172 and 173 in \cite{KL}). As for $n=4$, note that $H_{4,4}$ is the three-dimensional cube, which is the prism graph Pr$_4$. Thus,  $H_{4,4}$ is 3-word-representable by Proposition 15 in \cite{KP} and it is not 2-word-representable by Theorem 18 in \cite{K}. An example of 3-representation of $H_{4,4}$ given on page 90 in \cite{KL} is
$$414'343'231'12'24'1'3'44'2'33'11'22'.$$

\section{Concluding remarks}\label{conclusion}

In this paper we found the representation number of any crown graph solving at once two open problems in \cite{KL}. We suspect that crown graphs are (among) the hardest graphs to be represented in the class of bipartite graphs in the sense that they require longest words for representation. It would be interesting to prove or disprove this fact. In either case, one should be able to apply the methods used in this paper, in particular, the notion of a splittable set, to provide a complete classification for the representation number of any bipartite graph.

\end{document}